\documentclass[a4paper,oneside,8pt]{article}
\usepackage{amsmath}
\usepackage{amsthm,amssymb,enumerate,hyperref}
\usepackage[a4paper,margin=2cm,centering,nohead,ignorefoot,footskip=5ex]{geometry}
\usepackage{url}
\usepackage{xcolor}
\hypersetup{
    colorlinks,
    linkcolor={red!50!black},
    citecolor={red!50!black},
    urlcolor={blue!80!black}
}
\theoremstyle{plain}
\newtheorem{theorem}{Theorem}[section]
\newtheorem{proposition}[theorem]{Proposition}
\newtheorem{notation}[theorem]{Notation}
\newtheorem{lemma}[theorem]{Lemma}
\theoremstyle{definition}
\newtheorem{remark}[theorem]{Remark}
\newtheorem{definition}[theorem]{Definition}
\newtheorem{corollary}[theorem]{Corollary}
\newcommand{\N}{\mathbb{N}}

\newcommand{\norm}[1]{\left\lVert#1\right\rVert}

\begin{document}

\title{Quantitative results on the multi-parameters Proximal Point Algorithm }

\author{Bruno Dinis${}^{a}$ and Pedro Pinto${}^{b}$\\[2mm]
	\footnotesize ${}^{a}$ Departamento de Matem\'atica, Faculdade de Ci\^encias da
	Universidade de Lisboa,\\ 
	\footnotesize Campo Grande, Edif\'icio~C6, 1749-016~Lisboa, Portugal\\
	\footnotesize E-mail:  \protect\url{bmdinis@fc.ul.pt}\\[ 2mm]
	\footnotesize ${}^{b}$ Department of Mathematics, Technische Universit{\"a}t Darmstadt,\\ 
	\footnotesize Schlossgartenstrasse 7, 64289 Darmstadt, Germany \\
	\footnotesize E-mail:  \protect\url{pinto@mathematik.tu-darmstadt.de}
}
\maketitle

\begin{abstract}
We give a quantitative analysis of a theorem due to Fenghui Wang and Huanhuan Cui concerning the convergence of a multi-parametric version of the proximal point algorithm. Wang and Cui's result ensures the convergence of the algorithm to a zero of the operator. Our quantitative analysis provides explicit bounds on the metastability (in the sense of Terence Tao) for the convergence and the asymptotic regularity of the iteration. Moreover, our analysis bypasses the need of sequential weak compactness and only requires a weak form of the metric projection argument.
\end{abstract}

\section{Introduction}

In this paper we give a quantitative analysis of a theorem due to Fenghui Wang and Huanhuan Cui concerning the strong convergence of a multi-parametric version of the proximal point algorithm in Hilbert spaces. 

The \emph{proximal point algorithm} $(\mathsf{PPA})$ is recognized as a powerful and successful tool in approximating a zero of a maximal monotone operator in a Hilbert space. The algorithm was studied by Ralph Rockafellar in \cite{Rockafellar76}, where  weak convergence for $(\mathsf{PPA})$ was established.  A counter-example by  Osman G\"{u}ler in \cite{G(91)} showed that, in general, one cannot guarantee strong convergence for this iteration. This has prompted a series of variants in an attempt to obtain strong convergence.  Motivated by the success of the \emph{Halpern iterations} in fixed point theory \cite{Halpern67}, the \emph{Halpern-type proximal point algorithm} ($\mathsf{HPPA}$) was introduced by Shoji Kamimura and Wataru Takahashi in \cite{KT(00)} and, independently, by Hong-Kun Xu in \cite{X(02)}. With given points $u, z_0$, a regularization sequence  of positive real numbers $(c_n)$ and a sequence of errors $(e_n)$, ($\mathsf{HPPA}$) is recursively defined by
\begin{equation}\label{HPPA} \tag{\textsf{HPPA}}
z_{n+1}=\lambda_n u+(1-\lambda_n)J_{c_n}(z_n) +e_n,
\end{equation}
where $J_{c_n}$ is the resolvent function associated with $c_n$ and the maximal monotone operator. Strong convergence for \eqref{HPPA} was shown e.g.\ in  \cite[Theorem~2]{BM(11)} and \cite[Theorem~5.1]{X(02)}. These two strong convergence results received quantitative analyses in \cite{LLPP(ta)} and \cite{PP(ta)}, respectively. A generalization to Banach spaces was discussed in \cite{AM(17)}. This generalization received a quantitative analysis by Ulrich Kohlenbach in \cite{Koh(ta)}.

Yonghong Yao and Muhammad Aslam Noor studied in \cite{YaoNoor2008} a generalization of $\eqref{HPPA}$, in an attempt to obtain a result of strong convergence, in Hilbert spaces, under weaker assumptions. This generalization involves the use of several parameters and is called the \emph{multi-parameters proximal point algorithm} $\eqref{PPA}$. This was partially achieved in \cite[Theorem~3.3]{YaoNoor2008}, however a new condition was necessary that prevented the reduction to \eqref{HPPA}. A metastable version of this result was given in \cite{DP(ta)}. 

In this paper we give a quantitative analysis of a strong convergence result for \eqref{PPA}  by Wang and Cui \cite[Theorem~1]{WangCui2012}. Wang and Cui's result can be viewed as a generalization of previous results (see e.g.\ \cite{KT(00), X(02), MX(04), YaoNoor2008, BM(11)1}). Indeed, it relies on weaker conditions and enables a reduction to \eqref{HPPA}. The output of our analysis consists of explicit bounds on metastability properties (in the sense of Terence Tao \cite{T(08a),T(08b)}). Namely, we obtain functionals $\rho$ and $\widetilde{\rho}$ such that for every natural number $k$ and function $f:\N \to \N$ the following  properties hold
\begin{equation}\label{rho1}
\exists n \leq \rho(k,f) \, \forall i,j \in [n,f(n)] \left(\norm{z_i-z_j}\leq \frac{1}{k+1} \right),
\end{equation}
\begin{equation}\label{rho2}
\exists n \leq \widetilde{\rho}(k,f) \, \forall i \in [n,f(n)] \left(\norm{J_{c_i}(z_i)-z_i}\leq \frac{1}{k+1} \right),
\end{equation}
with the sequence $(z_n)$  defined by \eqref{PPA}. The propeties  \eqref{rho1} and \eqref{rho2} are the \emph{metastable} versions of the Cauchy property  and the asymptotic regularity for $(z_n)$, respectively. Notice that the original properties and their metastable versions are, in fact, (ineffectively) equivalent. While in general one cannot guarantee rate extraction for such properties, the underlying theoretical techniques ensure that we are always able to extract information for the corresponding metastable versions. For a discussion on the history and relevance of the notion of metastability we refer to \cite{K(18)}.

In this analysis, and similarly to previous analyses (cf. \cite{DP(ta),LLPP(ta),FFLLPP(19),PP(ta)}), we were guided by Fernando Ferreira and Paulo Oliva's \emph{bounded functional interpretation} \cite{FO(05)}, more specifically the classical variant from \cite{F(09)}. Functional interpretations are helpful to navigate the original proof, to avoid certain non-essential principles used therein (such as sequential weak compactness) and to obtain explicit bounds. The use of functional interpretations to analyse mathematical proofs has been very successful, particularly in areas such as approximation theory, ergodic theory, fixed point theory and optimization theory. We refer to \cite{K(17)} and the book \cite{K(08)} for an overview of such results. We would like to point out that, even though a proof-theoretical technique underlines the analysis in this paper, no knowledge of Mathematical Logic is required for the understanding of its results.

We work only with a weaker version of the metric projection principle where the projection point, crucial in the original proof, is replaced by approximations. Also, we bypass the sequential weak compactness arguments used in the original proof. The way to deal with the projection argument and sequential weak compactness is explained in full detail in \cite{FFLLPP(19)} (these qualitative improvements first appeared in \cite{K(11)}). Furthermore, the original proof has a discussion by cases which, in our quantitive analysis, imposes a discussion by cases for each approximation to the projection point. Namely, for each natural number $k$ and function $f$, one must consider a ``good enough'' approximation to the projection point and carry out the discussion by cases relative to that point. 

The structure of the paper is the following. In Section~\ref{sectionPrelim} we recall the relevant terminology as well as some well-known results from the theory of monotone operators in Hilbert spaces. We also recall some results necessary for our analysis. Also, in Subsection~\ref{sectionoriginalproof} we give a detailed description  of the original proof by Wang and Cui in order to clarify the different steps that our analysis requires. The main analysis is carried out in Section~\ref{Sectionanalysis}. As in the original proof we divide it into two cases depending on whether a certain auxiliary sequence is eventually decreasing or not. Some final remarks are left  to Section~\ref{sectionremarks}.

\section{Preliminaries}\label{sectionPrelim}

\subsection{Background on Monotone Operators on Hilbert spaces}
Throughout we let $H$ be a real Hilbert space with inner product $\langle \cdot, \cdot \rangle$ and norm $\norm{\cdot}$. We recall that an operator $\mathsf{A}:H \to 2^{H}$  is said to be \emph{monotone} if and only if whenever $(x,y)$ and $(x',y')$ are elements of the graph of $\mathsf{A}$, it holds that $\langle x-x',y-y'\rangle \geq 0$. A monotone operator $\mathsf{A}$ is said to be \emph{maximal monotone} if  the graph of $\mathsf{A}$ is not properly contained in the graph of any other monotone operator on $H$. We define $S:=\mathsf{A}^{-1}(0)$, the set of all \emph{zeros} of $\mathsf{A}$. For a comprehensive introduction to convex analysis and the theory of monotone operators in Hilbert spaces we refer to \cite{BC(17)}.

We fix $\mathsf{A}$ a maximal monotone operator and assume henceforth $S$ to be nonempty.
For every positive real number $\sigma$, we use $J_\sigma$  to denote the \emph{resolvent function} of $\mathsf{A}$, i.e.\ the single-valued function defined by $J_\sigma = (I + \sigma\mathsf{A} )^{-1}.$

\begin{definition}
A mapping $T : H \to H$ is called \emph{nonexpansive} if 
$$\forall x, y \in H \left(\norm{T (x) -T (y)}\leq \norm{x -y}\right),$$ 
and \emph{firmly nonexpansive} if
$$\forall x, y \in H\left( \norm{T(x)-T(y)}^2 \leq \norm{x-y}^2-\norm{(Id -T)(x)- (Id-T)(y)}^2 \right).$$
\end{definition}
Note that if $T$ is firmly nonexpansive then it is nonexpansive. The set $\{x \in H: T(x)=x\}$ of fixed points of the mapping $T$ will be denoted by  $\mathrm{Fix}(T)$.  If $T$ is nonexpansive, then $\mathrm{Fix}(T)$ is a closed and convex subset of $H$. 
 For $\sigma>0$, the resolvent function $J_\sigma$ is firmly nonexpansive and the set of fixed points of $J_\sigma$ is $S$.

Consider the following multi-parametric version of the proximal point algorithm introduced in \cite{YaoNoor2008},

\begin{equation}\label{PPA}\tag{\textsf{mPPA}}
z_{n+1}=\lambda_nu+\gamma_nz_n+\delta_nJ_{c_n}(z_n)+e_n, 
\end{equation}
where $u,z_0 \in H$ are given, $(c_n) \subset (0,+\infty) $, $(\lambda_n),(\delta_n) \subset (0,1)$ and $(\gamma_n) \subseteq [0,1)$ such that$\lambda_n+\gamma_n+\delta_n=1$, for all $n \in \N$.

Sometimes it is useful to consider the following \emph{exact} version of the algorithm \eqref{PPA}, 	
\begin{equation}\label{exactPPA}\tag{\textsf{mPPA}$_{\mathsf{e}}$}
y_{n+1}=\lambda_nu+\gamma_ny_n+\delta_nJ_{c_n}(y_n),
\end{equation}
where $u,y_0 \in H$ are given, $(c_n) \subset (0,+\infty)$, $(\lambda_n),(\delta_n) \subset (0,1)$ and $(\gamma_n) \subseteq [0,1)$ such that$\lambda_n+\gamma_n+\delta_n=1$, for all $n \in \N$.

Since we will only look at \eqref{exactPPA} as a way to prove strong convergence for \eqref{PPA}, we will always assume that $y_0=z_0$.

The following lemmas are well-known.

\begin{lemma}[Resolvent identity]
For $a,b>0$, the identity
\begin{equation*}
J_a (x) = J_b\left(\frac{b}{a}x+ \left(1-\frac{b}{a} \right)J_a (x)\right),
\end{equation*}
holds for every $x \in H$.
\end{lemma}

\begin{lemma}[\cite{MX(04)}]\label{lemmaresolvineq}
If $0<a\leq b$, then $\norm{J_a (x)-x} \leq 2 \norm{J_b (x)-x}$, for all $x \in H$.
\end{lemma}

\begin{lemma}\label{Lemmabasic}
Let $x, y \in H$ and let $t, s \geq 0$. Then
\begin{enumerate}
\item $\norm{x+y}^2\leq \norm{x}^2+2 \langle y,x+y \rangle$;
\item $\norm{tx+sy}^2=t(t+s)\norm{x}^2+s(t+s)\norm{y}^2-st\norm{x-y}^2$.
\end{enumerate}
\end{lemma}

We will use the following result due to Xu.

\begin{lemma}[\cite{X(02)}]\label{LemmaXu}
Let $(\alpha_n) \subset (0,1)$ and $(b_n)$ be real sequences such that
\begin{enumerate}[$(i)$]
\item $\sum \alpha_n = \infty$.
\item $\lim \alpha_n =0$.
\item $\limsup b_n \leq 0$ or $\sum \alpha_n |b_n|<\infty$.
\end{enumerate}
 Let $(a_n)$ be a nonnegative real sequence satisfying $a_{n+1}\leq (1-\alpha_n)a_n+ \alpha_nb_n$. Then $\lim a_n = 0$.
\end{lemma}

In this paper we carry out a quantative analysis of Theorem~\ref{ThmWangCui} below, due to Wang and Cui, which relies on the following conditions
\begin{enumerate}[($C_1$)]
\item\label{C1} $\lim \lambda_n=0$,
\item\label{C2} $\sum_{n=0}^{\infty} \lambda_n=\infty$,
\item\label{C3} $\liminf c_n>0$,
\item\label{C4} $\liminf\delta_n>0$,
\item\label{C5} $\sum_{n=1}^{\infty}\norm{e_n}<\infty$ or $\lim\frac{\norm{e_n}}{\lambda_n}=0$.
\end{enumerate}

\begin{theorem}{\rm (\cite[Theorem 1]{WangCui2012})}\label{ThmWangCui}
Let $(z_n)$ be generated by \eqref{PPA}. Assume that conditions $(C_{\ref{C1}})-(C_{\ref{C5}})$ hold. 
Then $(z_n)$ converges strongly to a point $z \in S$  (the nearest point projection of $u$ onto $S$).
\end{theorem}

\subsection{Quantitative Lemmas}

We recall the notion of monotone functional for two particular cases.
First consider the strong majorizability relation $\leq^{\ast}$ from \cite{bezem1985strongly} for functions $f,g:\N\to\N$.
$$g\leq^* f := \forall n, m\in\N\,\left(m\leq n\to \left( g(m)\leq f(n) \land f(m)\leq f(n)\right)\right).$$
A function $f:\N\to\N$ is said to be \emph{monotone} if  $f\leq^* f$, which corresponds to saying that $f$ is an increasing function, i.e.\ $\forall n\in\N\, \left(f(n)\leq f(n+1)\right)$. We say that a functional $\varphi:\N\times \N^{\N}\to\N$ is \emph{monotone} if for all $m,n\in\N$ and all $f,g:\N\to\N$,
$$\left(m\leq n \land g\leq^* f\right) \to \left(\varphi(m,g)\leq \varphi(n,f)\right).$$
A function depending on several variables (ranging over $\N$ or over $\N^\N$) is said to be monotone if it is monotone in all the variables.

\begin{remark}\label{maj}
We usually restrict our arguments to monotone functions in $\N^\N$. There is no real restriction in doing so, as for $f: \N \to \N$, one has $f \leq^* \! f^{\mathrm{maj}}$, where $f^{\mathrm{maj}}$ is the monotone function defined by $f^{\mathrm{maj}}(n):= \max\{f(i)\, :\, i \leq n\}$. In this way, we avoid constantly having to switch from $f$ to  $f^{\mathrm{maj}}$, and simplify the notation.
\end{remark}

\begin{notation}\label{notation}
	Consider a function $\varphi$ on tuples of variables $\bar{x}$, $\bar{y}$. If we wish to consider the variables $\bar{x}$ as parameters we write $\varphi[\bar{x}](\bar{y})$. For simplicity of notation we may then even omit the parameters and simply write $\varphi(\bar{y})$.
\end{notation}

We will use the following lemma.  

\begin{lemma}[\cite{LLPP(ta)}]\label{LemmaLLPP}
Let $(s_n)$ be a bounded sequence of non-negative real numbers, with $d\in\N \setminus\{0\}$ an upper bound for $(s_n)$, such that for any $n\in \N$
	\begin{equation*}
		s_{n+1}\leq (1-\alpha_n)s_n+\alpha_nr_n + \gamma_n,
	\end{equation*}
	\noindent where $(\alpha_n)\subset [0,1]$, $(r_n)$ and $(\gamma_n)\subset [0,+\infty)$ are given sequences of real numbers.\\	
	Assume that exist functions $A$, $R$, $G:\N \to \N$ such that
	\begin{itemize}
		\item[$(i)$] $\forall k\in \N \, \left( \sum\limits_{i=1}^{A(k)} \alpha_i \geq k \right)$,
		\item[$(ii)$] $\forall k\in \N \, \forall n\geq R(k) \, \left( r_n \leq \dfrac{1}{k+1} \right)$,
		\item[$(iii)$] $\forall k \in \N \, \forall n \in \N \, \left( \sum\limits_{i=G(k)+1}^{G(k)+n} \gamma_i \leq \dfrac{1}{k+1} \right)$.
	\end{itemize}
	Then
	\begin{equation*}
		\forall k \in \N \,\forall n\geq \theta(k) \, \left(s_n\leq \dfrac{1}{k+1}\right),
	\end{equation*}
	\noindent with $\theta(k):=\theta[A, R, G, d](k):=A(N-1+\lceil \ln(3d(k+1))\rceil)+1$, where $N:=\max\{ R(3k+2), G(3k+2)+1 \}$.
	\end{lemma}

It is well-known that for a sequence $(\alpha_n) \subset (0,1)$, having $\sum \alpha_n = \infty$ is equivalent to $\prod (1-\alpha_n)=0$. An alternative version of Lemma~\ref{LemmaLLPP} can be given where one assumes the existence of a rate of convergence $A'$ for the product $\prod (1-\alpha_n)$ instead of a rate of divergence $A$ for the sum $\sum \alpha_n $ (see \cite{LLPP(ta)}  and \cite[Lemma~2.4]{Korn}).

\subsection{The proof by Wang and Cui}\label{sectionoriginalproof}

Let us discuss the proof of Theorem~\ref{ThmWangCui} in detail in order to better understand the required steps of our quantitative analysis. We write $J_n$ to denote $J_{c_n}$, for $n\in\N$.
 \begin{enumerate}[$1)$]
\item The proof starts by showing that $\norm{z_n - y_n}\to 0$, using Lemma~\ref{LemmaXu}. This allows to reduce the convergence of a sequence $(z_n)$ given by \eqref{PPA} to that of a sequence $(y_n)$  given by the exact variant \eqref{exactPPA}.

\item By an easy induction argument it is shown that $(y_n)$ is bounded. 

\item Using Lemma~\ref{Lemmabasic} and the fact that the resolvent functions are firmly nonexpansive it is shown that 
\begin{equation}
\sigma \norm{J_n(y_n)-y_n}^2\leq M \lambda_n+ s_n-s_{n+1},
\end{equation}
where $\sigma, M$ are positive constants and $(s_n)$ is the sequence defined by $s_n:=\norm{y_n-\widetilde{z}\,}^2$, with $\widetilde{z}$ the projection point of $u$ onto $S$.
\item Form this point on, the proof proceeds by distinguishing the cases: $i)$ $(s_{n})$ is eventually decreasing, $ii)$ $(s_{n})$ is not eventually decreasing. In each case it is shown that $(s_n) \to 0$, which entails the result. 

Let us describe how the proof proceeds in each case.
 \begin{enumerate}
\item[$i)$] \underline{$(s_{n})$ is eventually decreasing.}
\begin{enumerate}[$a)$]
\item Since $(y_n)$ is bounded we have that $(s_n)$ is also bounded and therefore it is convergent.
\item By step 3 and the fact that $\lambda_n \to 0$ it folllows that $\norm{J_n(y_n)-y_n}\to 0$.
\item From the previous step and Lemma~\ref{lemmaresolvineq} it follows that $\norm{J_{\sigma}(y_n)-y_n}\to 0$.
\item It is shown that $s_{n+1}\leq (1-\lambda_n)s_n+2 \lambda_n \langle u-\widetilde{z},y_{n+1}-\widetilde{z}\, \rangle$.
\item By sequential weak compactness, using the demiclosedness principle \cite{B(65)} and the fact that $\widetilde{z}$ is the projection point, it follows that $\limsup \langle u-\widetilde{z},y_{n+1}-\widetilde{z} \, \rangle \leq 0$.

\item By Lemma~\ref{LemmaXu} one concludes that $s_n \to 0$.
\end{enumerate} 
\item[$ii)$]  \underline{$(s_{n})$ is not eventually decreasing.}
\begin{enumerate}[$a)$]
\item For a certain sequence $\tau(n)$ and $n_0 \in \N$, we have $s_{\tau(n)}\leq s_{\tau(n)+1}$, for $n \geq n_0$. By step 3, it holds that 
\begin{equation*}
\sigma\norm{J_{\tau(n)}(y_{\tau(n)})-y_{\tau(n)}}^2 \leq M\lambda_{\tau(n)}.
\end{equation*}
The sequence $\tau(n)$ is obtained using \cite[Lemma~3.1]{Mainge}.
\item Since $\lambda_{n}\to 0$ and $\tau(n)\to \infty$ one obtains that $\norm{J_{\tau(n)}(y_{\tau(n)})-y_{\tau(n)}}\to 0$.
\item From the previous step using sequential weak compactness and the demiclosedness principle, it follows that   $\limsup \langle u-\widetilde{z},y_{\tau(n)}-\widetilde{z}\, \rangle\leq 0.$
\item The previous step implies that $\limsup \langle u-\widetilde{z},y_{\tau(n)+1}-\widetilde{z}\, \rangle\leq 0$.
\item Since $s_{\tau(n)}\leq 2 \langle u-\widetilde{z},y_{\tau(n)+1}-\widetilde{z} \,\rangle$ it follows that $\limsup s_{\tau(n)}\leq 0$ and since $(s_{n})\subset [0,+\infty)$ one concludes that $\lim s_{\tau(n)}=0$ and consequentely $\lim s_{\tau(n)+1}=0$.
\item The result follows because $s_n \leq s_{\tau(n)+1}$. 
\end{enumerate}
\end{enumerate} 
\end{enumerate}


\section{Quantitative analysis}\label{Sectionanalysis}

We start by stating our quantitative assumptions. We assume that there exist $c \in \N\setminus \{0\}$ and monotone functions $\ell,L, E: \N \to \N$ and $h: \N\to \N\setminus \{0\}$ such that 

\begin{enumerate}[($Q_1$)]
\item\label{Q0} $\forall n \in \N\left(\lambda_n \geq \frac{1}{h(n)} \right)$,
\item\label{Q1} $\forall k \in \N \,\forall n \geq \ell(k) \left(\lambda_n \leq \frac{1}{k+1}\right)$,
\item\label{Q2} $\forall k \in \N \left(\sum_{i=1}^{L(k)}\lambda_i \geq k\right)$,
\item\label{Q3} $\forall n \in \N \left(\min\{c_n, \delta_n^2\} \geq \frac{1}{c}\right)$,
\item[($Q_{5a}$)]\label{Q5} $\forall k \in \N \, \forall n \in \N \left(\sum_{i=E(k)+1}^{E(k)+n}\norm{e_i}\leq \frac{1}{k+1} \right)$,
\item[($Q_{5b}$)]\label{Q5b} $\forall k \in \N \, \forall n \geq E(k) \left(\frac{\norm{e_n}}{\lambda_n}\leq \frac{1}{k+1} \right)$.
\end{enumerate}

The first condition is a quantitative version of the fact that the sequence $(\lambda_n)$ is always positive.
Condition $(Q_{\ref{Q1}})$ states that $\ell$ is a rate of convergence for the sequence $(\lambda_n)$. Condition $(Q_{\ref{Q2}})$ postulates that $L$ is a rate of divergence for the series $\sum_{n=0}^{\infty} \lambda_n$. Condition ($Q_{\ref{Q3}}$) expresses the fact that the terms of the sequences $(c_n)$ and $(\delta_n)$ are bounded away from zero. Finally, the last two conditions express, respectively, that the sequence of the partial sums $\left(\sum_{i=0}^n \norm{e_i}\right)$ is a Cauchy sequence with rate $E$, or that the sequence $\left(\frac{\norm{e_n}}{\lambda_n}\right)$ converges towards zero with rate of convergence $E$.

In the following we will assume, unless stated otherwise, that we are under the conditions $(Q_1)-(Q_4)$ and either $(Q_{5a})$ or $(Q_{5b})$.

\begin{notation}
In order to make the notation less cumbersome we will write $J_n$ instead of $J_{c_n}$ and $J$ instead of $J_{\frac{1}{c}}$.
\end{notation}

The following functions are useful for our analysis.

\begin{definition}\label{definitionfunctions1}
	We define functions $\zeta,  \overline{\sigma}, \phi_1,\phi_2,\widetilde{f}, r_1,r_2,r_3, r_4$ and $\Phi$, as follows.
	\begin{enumerate}[$(i)$]
	\item Given  $c\in\N$  and $\mathsf{C}: \N \to \N$,  for all $k, n\in \N,$
	 $$\zeta(k,n):=\zeta[c,\mathsf{C}](k,n):= c(k+1)\mathsf{C}(n)-1 \text{ (cf. Lemma~\ref{samefix-pt-set})}.$$
	\item Given $D\in\N$  and $L: \N \to \N$,  for all $k, n\in \N$,
	$$ \overline{\sigma}(k,n):=\overline{\sigma}[L,D](k,n):=L\left(n+\lceil \ln(12D^2(k+1))\rceil\right)+1 \text{ (we have $\overline{\sigma}[L,D]=\sigma[L,4D^2]$, cf. Lemma~\ref{lemmaqtXu1})}.$$
	\item Given $D\in\N$, for all $k, n\in \N$ and $f: \N \to \N$,
	 $$\phi_1(k,n,f):=\phi_1[D](k,n,f):=f\left(\phi_2(k,n,f)\right) \text{, with $\phi_2$ as defined below (cf. also Lemma~\ref{lemmaJyi})}.$$
	\item Given $D\in\N$, for all $k, n\in \N$ and $f: \N \to \N$,
	$$\phi_2(k,n,f):=\phi_2[D](k,n,f):= \max \{n, f^{(4D^2(k+1))}(n)\} \text{ (cf. Lemma~\ref{lemmaJyi})}.$$
	\item Given $k, D\in \N$ and $f, L: \N \to \N$, for all $n\in \N$,
	 $$\widetilde{f}(n):=\widetilde{f}[k, f, L, D](n):=f(\overline{\sigma}(k,n)).$$
	\item Given $c\in\N$, for all $n \in \N$,
	$$ r_1(n):=r_1[c](n):= 12c(n+1)^2-1.$$ 
	\item Given $k,D \in \N$, for all $n \in \N$,
	$$r_2(n):=r_2[k, D](n):=\max\{2(n+1),128D(k+1)^2\}.$$ 
	\item Given $k, c, D \in \N$  and $\ell : \N \to \N$,  for all $n\in\N$,
	$$r_3(n):=r_3[k, c, \ell, D](n):=\ell\left(\max \{96cD^2(n+1)^2-1,256D^2(k+1)^2-1, 16c(r_2(n))^2D^2\}\right).$$
	\item Given $k, c, D \in \N$  and $f, \ell, L: \N\to \N$, for all $n\in\N$,
	$$r_4(n):=r_4[k,c,f,\ell, D,L](n) :=3(k+1)(\widetilde{f}(\phi_2(r_1(n),r_3(n),\widetilde{f}+2))+1).$$
	\item Given $k, c, D\in\N$  and  $f, \ell, L: \N \to \N$, for all $n\in\N$,
	$$\Phi(n):=\Phi[k, c, f, \ell, L, D](n):= \phi_1(r_1(n),r_3(n),\widetilde{f}+2).$$
	\end{enumerate}
\end{definition}

Using the functions from Definition~\ref{definitionfunctions1}, we can now present the main functions.
\begin{definition}\label{definitionfunctions2}
	Given natural numbers $k, c, D\in \N$ and functions $f, \mathsf{C}, \ell, L:\N\to\N$, we define for all $n\in\N$,
	\begin{equation*}
	\Xi_1(n):=\Xi_1[k,f,c, \mathsf{C},\ell, L, D](n):=\zeta(2(6D+1)\max\{r_1, r_4 \}-1, \phi_1(r_1, r_3, \widetilde{f}+2)),
	\end{equation*}
	abbreviating $r_1=r_1(n), r_3=r_3(n)$ and $r_4=r_4(n)$.\\
	Given natural numbers $k, c, D \in \N$ and functions $f,c, \mathsf{C},\ell, L, h:\N\to\N$, we define for all $n\in\N$,
	\begin{equation*}
	\xi(n) := \xi[k, f,c, \mathsf{C},\ell, L, D, h](n):=\zeta(\max \{16h(f(n))(k+1)^2(6D+1),4c(r_2)^2 (6D+1)\},f(n)),
	\end{equation*}
abbreviating $r_2=r_2(n)$, and also
	\begin{equation*}
	\Xi_2(n) := \Xi_2[k, f,c, \mathsf{C},\ell, L,D, h](n):=\xi(\Phi(n)),
	\end{equation*}
	For every natural number $n\in\N$, we consider $\Xi(n):=\max\{ \Xi_1(n), \Xi_2(n)\}$.
\end{definition}

\begin{remark}\label{remrakmonotone}
	It is easy to check that all the functions defined in Definitions~\ref{definitionfunctions1} and \ref{definitionfunctions2}, except $\phi_1$ and $\phi_2$, are monotone, provided that the parameter functions are also  monotone. For $\phi_2$ we always have monotonicity in $n,f$. In order to obtain also monotonicity in $k$ it is enough that $f(n) \geq n$. Similarly for $\phi_1$.
\end{remark}

We are now able to formulate our main result. We then show some easy consequences, in particular a metastable version of Theorem~\ref{ThmWangCui} (cf. Corollary~\ref{cor_metazn}). For each point $z$, below $(s_{n}^{z})$ denotes the auxiliary sequence defined by $\norm{y_n-z}^2$.   
\begin{theorem}\label{theoremwangcui}
Let $(z_n)$ be generated by \eqref{PPA}. Assume that there exist $c \in \N\setminus\{0\}$ and monotone functions  $h:\N \to \N \setminus \{0\}$ and $\ell,L, E:\N \to\N$ such that conditions $(Q_{\ref{Q0}})-(Q_{\ref{Q5}})$ and either $(Q_{5a})$ or $(Q_{5b})$ hold. Let $D \in \N \setminus \{ 0 \}$ be such that $D\geq \max\{2 \norm{u-p},\norm{z_0-p} \}$, for some $p \in S$. Let $\mathsf{C}: \N \to \N$ be such that $c_n \leq \mathsf{C}(n)$, for all $n \in \N$. 
Then for all $k \in \N$ and monotone function $f:\N\to \N$
\begin{equation*}
\exists n \leq \mu(k,f) \, \exists z \in B_D \, \forall i \in [n,f(n)] \left(s_{i}^{z}\leq \frac{1}{k+1} \right),
\end{equation*}
where $\mu(k,f):=\max\{\overline{\sigma}(k, \phi_2(\bar{r},\bar{n},\widetilde{f}+2)),\Phi(\beta(\overline{k}, \Xi))\}$,  with $\bar{r}:=r_1(\beta(\overline{k},\Xi))$, $\bar{n}:=r_3(\beta(\overline{k},\Xi))$, $\overline{k}:=32(k+1)^2-1$ and $\overline{\sigma}, \widetilde{f},r_1, r_3,\phi_2,\Phi$ as in Definition~\ref{definitionfunctions1}, $\Xi(m)$ as in Definition~\ref{definitionfunctions2}, with $\beta$ as in Proposition~\ref{lemmaprojectarg}.
\end{theorem}

In the conditions of Theorem~\ref{theoremwangcui}, we have the following corollaries exhibiting, respectively, a metastability bound and a metastable version of asymptotic regularity for the iteration \eqref{exactPPA}.
\begin{corollary}\label{cor_metayn}
For all $k \in \N$ and monotone function $f:\N\to \N$,
\begin{equation*}
\exists n \leq \mu(4(k+1)^2-1,f) \, \forall i,j \in [n,f(n)] \left(\norm{y_i-y_j}\leq \frac{1}{k+1} \right).
\end{equation*}
\end{corollary}

\begin{proof}
	By Theorem~\ref{theoremwangcui} there exists $n \leq \mu(4(k+1)^2-1,f)$ and $z \in B_{D}$ such that for all $i \in [n,f(n)]$, $\norm{y_i-z}\leq \frac{1}{2(k+1)}$.
	Hence, for $i,j \in [n,f(n)]$ we have that $\norm{y_i-y_j}\leq \norm{y_i-z}+\norm{y_j-z}\leq \frac{1}{k+1}$.
\end{proof}

\begin{corollary}\label{cor_metaJi}
For all $k \in \N$ and monotone function $f:\N\to \N$, we have
\begin{enumerate}[$(i)$]
\item\label{cor_metJi1} $\exists n \leq \widetilde{\mu}(k,f) \, \forall i \in [n,f(n)] \left(\norm{J_i(y_i)-y_i}\leq \frac{1}{k+1} \right)$,
\item\label{cor_metJi2} $\exists n \leq \widetilde{\mu}(2k+1,f) \, \forall i \in [n,f(n)] \left(\norm{J(y_i)-y_i}\leq \frac{1}{k+1} \right)$,
\end{enumerate}
where $\widetilde{\mu}(k,f):=\max\{\mu(16c^2(k+1)^2-1,\check{f}+1),\ell(4cD(k+1)-1)\}$ and $\check{f}(m):=f(\max\{m, \ell(4cD(k+1)-1)\})$.
\end{corollary}

\begin{proof}
It follows from Corollary~\ref{cor_metayn} that there exists $n_0 \leq \mu(16c^2(k+1)^2-1,\check{f}+1)$ such that $$\forall i \in [n_0,\check{f}(n_0)]\left(\norm{y_{i+1}-y_{i}}\leq \frac{1}{2c(k+1)} \right).$$ Let $n:=\max\{n_0,\ell(4cD(k+1)-1)\}$. Observe that $n\leq \widetilde{\mu}(k,f)$ and $[n,f(n)] \subseteq [n_0,\check{f}(n_0)]$. Then, for $i \in [n,f(n)]$ we have that $\norm{y_{i+1}-y_i}\leq \frac{1}{2cD(k+1)}$ and $\lambda_i\leq \frac{1}{4cD(k+1)}$, by condition $(Q_2)$.

We have that 
\begin{equation*}
\begin{split}
\norm{J_i(y_i)-y_i} &\leq  \norm{J_i(y_i)-(\lambda_i u + \gamma_i y_i+ \delta_iJ_i(y_i))}+\norm{y_{i+1}-y_i} \\
& \leq  \lambda_i \norm{J_i(y_i)-u}+\gamma_i \norm{J_i(y_i)-y_i}+\norm{y_{i+1}-y_i} .
\end{split}
\end{equation*}

Hence $(1-\gamma_i)\norm{J_i(y_i)-y_i} \leq \lambda_i \norm{J_i(y_i)-u}+\norm{y_{i+1}-y_i}$. Since

\begin{equation*}
\frac{1}{c}\norm{J_i(y_i)-y_i} \leq \delta_i\norm{J_i(y_i)-y_i} \leq (\lambda_i+\delta_i)\norm{J_i(y_i)-y_i}=(1-\gamma_i)\norm{J_i(y_i)-y_i},
\end{equation*}

and $\norm{J_i(y_i)-u}\leq 2D$, we conclude that for $i \in [n,f(n)]$
\begin{equation*}
\norm{J_i(y_i)-y_i} \leq  c\lambda_i \norm{J_i(y_i)-u}+c \norm{y_{i+1}-y_i}\leq \frac{2cD}{4cD(k+1)}+\frac{c}{2c(k+1)}=\frac{1}{k+1}. 
\end{equation*} 
This conclude the proof of Part~\eqref{cor_metJi1}. Part~\eqref{cor_metJi2} then follows from Lemma~\ref{lemmaresolvineq}.
\end{proof}

In the original proof, the convergence of \eqref{PPA} is reduced to that of the exact variant \eqref{exactPPA}. The  quantitative version of that reduction is shown in Lemma~\ref{lemma_zn-yn_small} provided that the sequence $\norm{z_n-y_n}$ is bounded as shown in Lemma~\ref{aux_bounds}.

\begin{lemma}\label{aux_bounds}
	Consider a monotone function $E:\N\to\N$. Let $d_0,d_1, d_2\in \N\setminus\{0\}$ be natural numbers satisfying $d_0\geq \max\left\{\norm{u-p},\norm{z_0-p} \right\}$, $d_1\geq \sum_{i=0}^{E(0)}\norm{e_i}+1$, for some $p\in S$, and $d_2\geq\max\{\norm{z_i-y_i}\, :\, i\leq E(0)\}$.
	\begin{enumerate}[$(i)$]
    \item \label{raiodabola} $\forall n \in \N \left( \norm{y_{n}-p}\leq d_0\right)$,
	\item If $E$ satisfies $(Q_{5a})$, then $\forall n\in \N\, \left( \norm{z_n-y_n} \leq 2d_0+d_1\right)$,
	\item If $E$ satisfies $(Q_{5b})$, then $\forall n\in \N\, \left( \norm{z_n-y_n} \leq d_2\right)$.
	\end{enumerate}
\end{lemma}

\begin{proof}
	Since the resolvent is nonexpansive we have that
	\begin{equation}\label{eqladecima}
	\begin{split}
	\norm{z_{n+1}-y_{n+1}}&\leq \gamma_n\norm{z_n-y_n}+\delta_n\norm{J_{n}z_n-J_{n}y_n}+\norm{e_n}\\
	&\leq \gamma_n\norm{z_n-y_n}+\delta_n\norm{z_n-y_n}+\norm{e_n}\\
	&=(1-\lambda_n)\norm{z_n-y_n}+\norm{e_n}.
	\end{split}
	\end{equation}
	One sees that $\forall n \in \N \left( \norm{y_{n}-p}\leq d_0\right)$ by induction. Indeed, clearly $\norm{y_0-p}=\norm{z_0-p}\leq d_0$. As for the induction step we have
	\begin{equation*}
	\begin{split}
	\norm{y_{n+1}-p}&\leq \lambda_n\norm{u-p}+\gamma_n\norm{y_n-p}+\delta_n\norm{y_n-p}\\
	&=\lambda_n\norm{u-p}+(1-\lambda_n)\norm{y_n-p}\\
	&\leq \lambda_nd_0+(1-\lambda_n)d_0=d_0.
	\end{split}
	\end{equation*}
	Assume that $E$ satisfies $(Q_{5a})$. In particular, for $k=0$ we have $\forall n\in \N\left(\sum_{i=E(0)+1}^{E(0)+n}\norm{e_i}\leq 1 \right)$. 

\noindent For all $m\in\N$
	\begin{equation*}
	\sum_{i=0}^{m}\norm{e_i}\leq \sum_{i=0}^{E(0)}\norm{e_i}+\sum_{i=E(0)+1}^{m}\norm{e_i}\leq d_1.
	\end{equation*} 
	Then, similarly to \eqref{raiodabola}, one shows that $\norm{z_n-p}\leq d_0+ d_1$. Hence, for all $n\in \N$
	\begin{equation*}
	\norm{z_{n}-y_{n}}\leq\norm{z_n-p}+\norm{y_n-p}\leq 2d_0+d_1.
	\end{equation*}
	
	Assume now that $E$ satisfies $(Q_{5b})$. We show by induction that $\norm{z_n-y_n}\leq d_2$, for all $n \in \N$. Clearly $\norm{z_0-y_0}\leq d_2$. Assume that $\norm{z_n-y_n}\leq d_2$. If $n<E(0)$, then $n+1 \leq E(0)$ and therefore $\norm{z_{n+1}-y_{n+1}}\leq d_2$. For $n \geq E(0)$, we have $\norm{e_n}\leq \lambda_n$. Then by \eqref{eqladecima}, the induction hypothesis, and the fact that $d_2\geq 1$
	\begin{equation*}
	\norm{z_{n+1}-y_{n+1}}\leq (1-\lambda_n)d_2+\lambda_n\leq(1-\lambda_n)d_2+\lambda_nd_2= d_2.\qedhere
	\end{equation*}
\end{proof}

\begin{lemma}\label{lemma_zn-yn_small}
Let $(z_n)$, $(y_n)$ be given, respectively, by \eqref{PPA} and \eqref{exactPPA}. Consider monotone functions $L, E:\N\to\N$ such that $L$ satisfies $(Q_3)$ and $E$ satisfies either $(Q_{5a})$ or $(Q_{5b})$. Let $d_0, d_1, d_2\in \N\setminus\{0\}$ be natural numbers as in Lemma~\ref{aux_bounds}. Define $d:=\max\{2d_0+d_1,d_2\}$.  Then the sequence $(z_n-y_n)$ converges to zero and has rate of convergence $\Theta$, i.e.
\begin{equation*}
\forall k \in \N \, \forall n \geq \Theta(k)\left(\norm{z_n -y_n}\leq \frac{1}{k+1} \right),
\end{equation*}
where $\Theta(k):=\Theta[L, E, d_0, d_1, d_2](k):= L(E(3k+2)+\lceil\ln(3d(k+1))\rceil)+1$.
\end{lemma}

\begin{proof}
In the case where $E:\N \to \N$ satisfies $(Q_{5a})$, by Lemma~\ref{aux_bounds}, we have $\norm{z_n-y_n}\leq 2d_0+d_1$ for all $n\in\N$. By inequality \eqref{eqladecima} we can instantiate Lemma~\ref{LemmaLLPP} with $s_n=\norm{z_n-y_n}$, $\alpha_n=\lambda_n$, $r_n\equiv 0$, $\gamma_n=\norm{e_n}$, $A=L$, $R\equiv 0$ and $G=E$. Hence 
\begin{equation}\label{theta1}
\forall k \in \N \,\forall n \geq \theta_1(k)\left(\norm{z_n-y_n}\leq\frac{1}{k+1}\right),
\end{equation}
with $\theta_1(k):=L\left(E(3k+2)+\lceil \ln(3(2d_0+d_1)(k+1))\rceil \right)+1$.

In the case where $E:\N \to \N$ satisfies $(Q_{5b})$, by Lemma~\ref{aux_bounds}, we have $\norm{z_n-y_n}\leq d_2$ for all $n\in\N$. Instantiating Lemma~\ref{LemmaLLPP} with $s_n=\norm{z_n-y_n}$, $\alpha_n=\lambda_n$, $r_n=\frac{\norm{e_n}}{\lambda_n} $, $\gamma_n\equiv 0$, $A=L$, $R=E$ and $G\equiv 0$, 
\begin{equation}\label{theta2}
\forall k \in \N \, \forall n \geq \theta_2(k)\left(\norm{z_n-y_n}\leq\frac{1}{k+1}\right),
\end{equation}
with $\theta_2(k):=L\left(\max\{E(3k+2)-1,0\}+\lceil \ln(3d_2(k+1))\rceil \right)+1$.
The monotonicity of the function $L$ implies $\max\{\theta_1(k),\theta_2(k)\}\leq \Theta(k)$. From \eqref{theta1} and \eqref{theta2}, we conclude the result.
\end{proof}

In the conditions of Theorem~\ref{theoremwangcui} and Lemma~\ref{lemma_zn-yn_small}, we have the following corollary exhibiting a metastability bound for the iteration \eqref{PPA}.
\begin{corollary}\label{cor_metazn}
For all $k \in \N$ and monotone function $f:\N\to \N$,
\begin{equation*}
\exists n \leq \nu(k,f) \, \forall i,j \in [n,f(n)] \left(\norm{z_i-z_j}\leq \frac{1}{k+1} \right),
\end{equation*}
where $\nu(k,f):= \max\{\mu(36(k+1)^2-1,\widehat{f}), \Theta(3k+2)\}$, with $\widehat{f}(m):=f(\max\{m, \Theta(3k+2)\})$, $\mu$ as in Theorem~\ref{theoremwangcui} and $\Theta$ as in Lemma~\ref{lemma_zn-yn_small}.
\end{corollary}

\begin{proof}
	By Corollary~\ref{cor_metayn}, there exists $n_0 \leq \mu(36(k+1)^2-1,\widehat{f})$ such that for all $i,j \in \left[n_0,\widehat{f}(n_0)\right]$ it holds that $\norm{y_i-y_j}\leq \frac{1}{3(k+1)}$. Define $n:= \max\{n_0,\Theta(3k+2)\}\leq \nu(k,f)$. Then clearly $[n,f(n)]\subset [n_0,\widehat{f}(n_0)]$ and for $i \in [n,f(n)]$, we have $i \geq \Theta(3k+2)$.
	Using Lemma~\ref{lemma_zn-yn_small} we conclude that
	\begin{equation*}
	\norm{z_i-z_j}\leq \norm{z_i-y_i}+\norm{y_i-y_j}+\norm{y_j-z_j}\leq \frac{1}{k+1}.\qedhere
	\end{equation*} 
\end{proof}

\begin{remark}\label{rho1remark1}
	The functional $\rho$ defined by $\rho(k,f):=\nu(k,f^{maj})$ satisfies \eqref{rho1}, i.e. the restriction to monotone functions in Corollary~\ref{cor_metazn} poses no real limitation (cf. Remark~\ref{maj}). 
\end{remark}

Using Corollary~\ref{cor_metaJi} and Lemma~\ref{lemma_zn-yn_small} we obtain a metastable version of the asymptotic regularity for the iteration \eqref{PPA}.
\begin{corollary}\label{cor_metaJiz}
For all $k \in \N$ and monotone function $f:\N\to \N$, we have
\begin{enumerate}[$(i)$]
\item\label{cor_metJiz1} $\exists n \leq \widetilde{\nu}(k,f) \, \forall i \in [n,f(n)] \left(\norm{J_i(z_i)-z_i}\leq \frac{1}{k+1} \right)$,
\item\label{cor_metJiz2} $\exists n \leq \widetilde{\nu}(2k+1,f) \, \forall i \in [n,f(n)] \left(\norm{J(z_i)-z_i}\leq \frac{1}{k+1} \right)$,
\end{enumerate}
where $\widetilde{\nu}(k,f):=\max\{\widetilde{\mu}(2k+1,\breve{f}),\Theta(4k+3)\}$ and $\breve{f}(m):=f(\max\{m, \Theta(4k+3)\})$.
\end{corollary}

\begin{proof}
By Corollary~\ref{cor_metaJi} there exists $n_0 \leq \widetilde{\mu}(2k+1, \breve{f})$ such that 
\begin{equation*}
\forall i \in [n_0,\breve{f}(n_0)] \left( \norm{J_i(y_i)-y_i}\leq \frac{1}{2(k+1)} \right).
\end{equation*}

Let $n:=\max\{n_0, \Theta(4k+3)\}$. Then $n \leq \widetilde{\nu}(k,f)$ and for $i \in [n,f(n)] \subseteq [n_0, \breve{f}(n_0)]$ we have that 
\begin{equation*}
\begin{split}
\norm{J_i(z_i)-z_i}& \leq \norm{z_i-y_i} + \norm{y_i-J_i(y_i)}+ \norm{J_i(y_i)-J_i(z_i)}\\
& \leq 2\norm{z_i-y_i} + \norm{y_i-J_i(y_i)} \leq \frac{1}{k+1}.
\end{split}
\end{equation*}
This shows Part~\eqref{cor_metJiz1}. Part~\eqref{cor_metJiz2} then follows from Lemma~\ref{lemmaresolvineq}.
\end{proof}

\begin{remark}\label{rho1remark2}
Similarly to Remark~\ref{rho1remark1}, the functional $\widetilde{\rho}$ defined by $\widetilde{\rho}(k,f):=\widetilde{\nu}(k,f^{maj})$ satisfies \eqref{rho2}. 
\end{remark}

In the remainder of this section we carry out the analysis of Theorem~\ref{ThmWangCui} which provides a proof to Theorem~\ref{theoremwangcui}. We begin with a lemma relating the resolvent functions $J$ and $J_n$.

\begin{lemma}[\cite{DP(ta)}]\label{samefix-pt-set}Consider a monotone function $\mathsf{C}:\N\to \N$ such that $c_n\leq \mathsf{C}(n)$, for all $n\in\N$. For any $k,n\in \N$ and any $z\in H$,
	\[
	\norm{J(z)-z}\leq \frac{1}{\zeta(k,n)+1} \;\to\; \forall n'\leq n\, \left(\norm{J_{n'}(z)-z}\leq \frac{1}{k+1}\right),
	\]
	with $\zeta(k,n):=\zeta[c,\mathsf{C}](k,n):= c(k+1)\mathsf{C}(n)-1$.
\end{lemma}

\begin{notation}
	For $p \in S$  and $D \in \N$, we denote by $B_{D}$ the closed ball centered at $p$ with radius $D$, i.e.\ $
	B_{D}:=\{ z \in H: \norm{z-p}\leq D\}.$ In the following, a point $p$ is always made clear from the context.
\end{notation}

We recall the following quantitative result related to the projection argument.

\begin{proposition}[\cite{PP(ta)}]\label{lemmaprojectarg}
	Let $D\in \N\setminus \{0\}$ be such that $D\geq 2\norm{u-p}$ for some $p\in S$.\\
	For any natural number $k$ and monotone function $f:\N \to \N $, there are $n \leq \beta(k,f)$ and $z\in B_D$ such that
	$$\norm{J(z)-z} \leq \frac{1}{f(n)+1} \, \land \, \forall y\in B_D 
	\left(\norm{J(y)-y}\leq \frac{1}{n+1} \to \langle u-z,y-z\rangle \leq \frac{1}{k+1}\right),$$ 
	where $\beta(k,f):=\beta[D](k,f):=24D\left(w_{f}^{(R)}(0)+1\right)^2$,\\ with  $R:=R[D,k]:=4D^4(k+1)^2$ and $w_{f}(m):=w_f[D,f](m):=\max\{ f(24D(m+1)^2),\, 24D(m+1)^2 \}$.
\end{proposition}

The fact that we are working with almost fixed points instead of actual fixed points creates a new error term $P_{n}^{z}$ in the main inequalities (cf. \eqref{Desigmain} and \eqref{desigJnyn-yn} below). Since we can consider good enough almost fixed points $z$, this error $P_{n}^{z}$ can be made small so as not to affect the convergence of the algorithm.

\begin{lemma}\label{Lemmaoriginaleq10}
Let $D \in \N \setminus \{ 0 \}$ be such that $D\geq \max\{2 \norm{u-p},\norm{z_0-p} \}$, for some $p \in S$, and $c \in \N \setminus \{0\}$ satisfying $(Q_4)$. For all $n \in \N$ and  $z \in B_{D}$  we have 
\begin{equation}\label{Desigmain}
s^{z}_{n+1}\leq (1-\lambda_{n})(s^{z}_{n}+P_{n}^{z})+2\lambda_n \langle u-z, y_{n+1}-z\rangle 
\end{equation}
and
\begin{equation}\label{desigJnyn-yn}
\norm{J_{n}(y_n)-y_n}^2 \leq c\left(8D^2\lambda_n + s_{n}^z-s_{n+1}^{z}+P_{n}^{z}\right),
\end{equation}
where $s_{n}^{z}:=\norm{y_n-z}^2$, $P_{n}^{z}:=2\norm{J_{n}(z)-z}\left(3\norm{y_n-z}+\norm{J_{n}(z)-z} \right)$. 
\end{lemma}

\begin{proof}
Let $z$ be a point in $B_{D}$. Since the resolvent is  nonexpansive we have that
\begin{equation}\label{ineqJ}
\begin{split}
\norm{y_n-J_n(y_n)- (z-J_n(z))}^2 &\geq \left(\norm{J_n(y_n)-y_n}-\norm{J_n(z)-z}\right)^2\\
&= \norm{J_n(y_n)-y_n}^2 + \norm{J_n(z)-z}\left(-2 \norm{J_n(y_n)- y_n}+\norm{J_n(z)-z}\right)\\
&\geq \norm{J_n(y_n)-y_n}^2 +\\
&\norm{J_n(z)-z} \left(-2\left(\norm{J_n(y_n)-J_n(z)} + \norm{J_n(z)-z}+ \norm{y_n- z}\right)  +\norm{J_n(z)-z}\right)\\
& =\norm{J_n(y_n)-y_n}^2  - \norm{J_n(z)-z} \left(4 \norm{y_n- z} +\norm{J_n(z)-z}\right).
\end{split}
\end{equation}

Using \eqref{ineqJ} and the fact that the resolvent is both nonexpansive and firmly nonexpansive we derive 

\begin{equation*}
\begin{split}
\norm{J_{n}(y_n)-z}^2&\leq \left( \norm{J_{n}(y_n)-J_{n}(z)}+\norm{J_{n}(z)-z}\right)^2\\
&=\norm{J_{n}(y_n)-J_{n}(z)}^2+\norm{J_{n}(z)-z}\left(2\norm{J_{n}(y_n)-J_{n}(z)}+\norm{J_{n}(z)-z} \right)\\
&\leq \norm{y_n-z}^2-\norm{y_n-J_{n}(y_n)-z+J_{n}(z)}^2+\norm{J_{n}(z)-z}\left(2\norm{y_n-z}+\norm{J_{n}(z)-z}\right)\\
& \leq \norm{y_n-z}^2-\norm{J_{n}(y_n)-y_n}^2+2\norm{J_{n}(z)-z}\left(3\norm{y_n-z}+\norm{J_{n}(z)-z}\right).
\end{split}
\end{equation*}

Then, the definition of $y_{n+1}$ and Lemma~\ref{Lemmabasic} entail
\begin{equation*}\label{ineqtoapplemma}
\begin{split}
\norm{y_{n+1}-z}^2&\leq \norm{\gamma_n(y_n-z)+\delta_n\left(J_{n}(y_n)-z\right)}^2+2\lambda_n\langle u-z,y_{n+1}-z \rangle\\
&= \gamma_n(\gamma_n+\delta_n)\norm{y_n-z}^2+\delta_n(\gamma_n+\delta_n)\norm{J_{n}(y_n)-z}^2\\
& \qquad -\gamma_n\delta_n\norm{J_{n}(y_n)-y_n}^2+2\lambda_n\langle u-z,y_{n+1}-z \rangle\\
&\leq \gamma_n(\gamma_n+\delta_n)\norm{y_n-z}^2+\delta_n(\gamma_n+\delta_n)\Big[\norm{y_n-z}^2-\norm{J_{n}(y_n)-y_n}^2 \\
& \qquad  +2\norm{J_{n}(z)-z}\left(3\norm{y_n-z}+\norm{J_{n}(z)-z}\right)\Big]\\ 
&\qquad -\gamma_n\delta_n\norm{\left(J_{n}(y_n)-y_n \right)}^2+2\lambda_n\langle u-z,y_{n+1}-z\rangle\\
&\leq (1- \lambda_n)\norm{y_n-z}^2+2\lambda_n\langle u-z,y_{n+1}-z \rangle-\delta_n(2\gamma_n+\delta_n)\norm{J_{n}(y_n)-y_n}^2\\
&\qquad +2(1-\lambda_n)\norm{J_{n}(z)-z}\left(3\norm{y_n-z}+\norm{J_{n}(z)-z} \right).
\end{split}
\end{equation*} 
We conclude that \eqref{Desigmain} holds.
Also, since $\delta_n(2\gamma_n+\delta_n) \geq \delta_n^2 \geq \frac{1}{c}$,
\begin{equation}\label{desig10}
s_{n+1}^z-s_{n}^{z}+\lambda_n s_{n}^{z}+\frac{1}{c}\norm{J_{n}(y_n)-y_n}^2\leq 2 \lambda_n\langle u-z,y_{n+1}-z\rangle+P_{n}^{z}.
\end{equation}
The inequality \eqref{desigJnyn-yn} follows from the fact that $2\langle u-z,y_{n+1}-z\rangle \leq 2 \norm{u-z}\norm{y_{n+1}-z}\leq8D^2$ .
\end{proof}

To deal with the remainder of the analysis we must discuss two cases depending on whether the sequence $(s^{z}_n)$ is decreasing or not.

\subsection{First case}

The first case that we are going to consider is the case where the sequence $(s^{z}_n)$ is eventually decreasing. 
Our goal is to apply Lemma~\ref{lemmaqtXu1} below, which is a quantitative version of Lemma~\ref{LemmaXu}, with $(v_n):=P_{n}^{z}$, $(r_{n}):= 2 \langle u-z, y_{n+1}-z\rangle$, for an adequate choice of $z$, in order to obtain a rate of metastability for $(s^{z}_n)$. The result is an easy adaptation of \cite[Lemma~14]{PP(ta)} for the case where $(\gamma_n)\equiv 0$.
\begin{lemma}\label{lemmaqtXu1}
Let $(s_n)$ be a bounded sequence of non-negative real numbers and $M\in\N$ a positive upper bound on $(s_n)$. Consider sequences of real numbers $(\lambda_n)\subset (0,1)$, $(r_n)$ and $(v_n)$ and assume the existence of a monotone function $ L$ satisfying  condition $(Q_{\ref{Q2}})$. For natural numbers $k, n$ and $q$ assume
\[\forall i\in[n,q]\, \left(v_i\leq \frac{1}{3(k+1)(q+1)}\land r_i\leq \frac{1}{3(k+1)}\right),\]
and for all $i\in\N$,
\[s_{i+1}\leq (1-\lambda_i)(s_i+v_i)+\lambda_ir_i.\]
Then
\[\forall i\in[\sigma(k,n),q]\, \left(s_i\leq \frac{1}{k+1}\right),\]
with $\sigma(k,n):=\sigma[L, M](k,n):=L\left(n+\lceil \ln(3M(k+1))\rceil\right)+1$.
\end{lemma}

\begin{remark}\label{remarksigma}
Observe that since $\lambda_n \leq 1$, for all $n \in \N$, by $(Q_3)$ it follows that for all $n \in \N$ we have $L(n) \geq n$. Hence the function $\sigma$ defined in Lemma~\ref{lemmaqtXu1} verifies the condition $\sigma(k,n) \geq n$, for all $n \in \N$.
\end{remark}
 In the original proof of Theorem~\ref{ThmWangCui}, metric projection is used, followed by a sequential weak compactness argument.  Sequential weak compactness can be eliminated in a way similar to \cite{DP(ta),PP(ta)}. This is to be expected in light of the arguments given in \cite{FFLLPP(19)}. The next result is an easy adaptation of \cite[Proposition~2.27]{K(08)} (see also Remark~2.29 in the same reference).

\begin{lemma}\label{lemmaJyi}
	Let $D \in \N \setminus \{ 0 \}$ be such that $D\geq \max\{2 \norm{u-p},\norm{z_0-p} \}$, for some $p \in S$. For $k,n \in \N$, $f:\N \to \N$ monotone and $z \in B_{D}$, if
	\begin{equation*}
	\forall i \in [n, \phi_1(k,n,f)] \left(s_{i+1}^{z} \leq s_{i}^{z} \right),
	\end{equation*}
	then there exists $n' \leq \phi_2 (k,n,f)$ such that
	\begin{equation}\label{eqCauchymeta}
	 \forall i,j \in [n',f(n')] \left(\left|s_{i}^{z}-s_{j}^{z}\right|\leq \frac{1}{k+1}\right),
	\end{equation}
	where $\phi_1(k,n,f):=\phi_1[D](k,n,f):=f(\phi_2(k,n,f))$ and $\phi_2(k,n,f):=\phi_2[D](k,n,f):= \max\{n, f^{(4D^2(k+1))}(n)\}$.
	
	Moreover, there is $n'\in \{f^{(i)}(n): i \leq 4D^2(k+1)\}$ satisfying \eqref{eqCauchymeta}.
\end{lemma}

We will need the following particular instance of Lemma~\ref{lemmaJyi}.
\begin{lemma}\label{lemmametadifference}
		Let $D \in \N \setminus \{ 0 \}$ be such that $D\geq \max\{2 \norm{u-p},\norm{z_0-p} \}$, for some $p \in S$. For $k,n \in \N$, $f:\N \to \N$ monotone and $z \in B_{D}$, if
	\begin{equation*}
	\forall i \in [n, \phi_1(k,n,f+1)] \left(s_{i+1}^{z} \leq s_{i}^{z} \right),
	\end{equation*}
	then there exists $n' \leq \phi_2 (k,n,f+1)$ such that
	\begin{equation}\label{eqCauchymeta2}
	\forall i \in [n',f(n')] \left(s_{i}^{z}-s_{i+1}^{z}\leq \frac{1}{k+1}\right),
	\end{equation}
	where $\phi_1$, $\phi_2$ are as in Lemma~\ref{lemmaJyi}.

	Moreover, there is $n' \in \{(f+1)^{(i)}(n): i \leq 4D^2(k+1)\}$ satisfying \eqref{eqCauchymeta2}.
\end{lemma}

\begin{proof}
	We may assume that $f(n) \geq n$, because otherwise the result is trivial.
	By Lemma~\ref{lemmaJyi} we have that $$\forall i,j \in [n', f(n')+1] \left(\left|s_{i}^{z}-s_j^{z} \right|\leq \frac{1}{k+1} \right),$$ with $n' \in \{(f+1)^{(i)}(n): i \leq 4D^2(k+1)\}$. If $i \in [n',f(n')]$, then $i+1 \in [n', f(n')+1]$ and so $\left|s_{i}^{z}-s_{i+1}^{z} \right|\leq \frac{1}{k+1} $.
	Since $n' \geq n $ and $f(n')\leq f(\phi_2(k,n,f+1)) \leq \phi_1(k,n,f+1)$, using the monotonicity of the function $f$ we have $[n',f(n')] \subseteq [n,\phi_1(k,n,f+1)]$. Then for $i \in [n',f(n')]$ it holds that $s_{i}^{z}-s_{i+1}^{z} \geq 0$ which entails the result.
\end{proof}

In the discussion of the first case we need a quantitative version of the fact that $(y_n)$ is a sequence of almost fixed points for the resolvent function. This is accomplished with Lemmas~\ref{linnerproduct0} and \ref{lemma_g}.

\begin{lemma}\label{linnerproduct0}
	Let $D \in \N \setminus \{ 0 \}$ be such that $D\geq \max\{2 \norm{u-p},\norm{z_0-p} \}$, for some $p \in S$, and $c\in \N \setminus \{0\} $ satisfying $(Q_4)$. For $m,n \in \N$, $f:\N \to \N$ monotone and $z \in B_{D}$, if $n \geq \ell(96cD^2(m+1)^2-1)$ and
	\begin{equation*}
	 \forall i \in [n, f(n)] \left(s_{i}^{z}-s_{i+1}^{z}\leq \frac{1}{12c(m+1)^2} \wedge P_{i}^{z}\leq \frac{1}{12c(m+1)^2} \right),
	\end{equation*}
	then 
	\begin{equation*}
	\forall i \in [n, f(n)] \left(\norm{J(y_i)-y_i}\leq \frac{1}{m+1} \right).
	\end{equation*}
\end{lemma}

\begin{proof}
	For $i\in [n, f(n)]$, we have $i\geq n\geq \ell(96cD^2(m+1)^2-1)$. Hence, by condition (Q$_{\ref{Q1}}$),
	\begin{equation*}
	\lambda_{i}\leq \frac{1}{96cD^2(m+1)^{2}}.
	\end{equation*}
	By inequality \eqref{desigJnyn-yn}
	\begin{equation*}
	\norm{J_{i}(y_i)-y_i}^2 \leq c\left(\frac{8D^2}{96cD^2(m+1)^{2}} + \frac{1}{12c(m+1)^2}+\frac{1}{12c(m+1)^2}\right) = \frac{1}{4(m+1)^2}.
	\end{equation*}
	Hence
	\begin{equation*}
	\forall i \in [n,f(n)] \left(\norm{J_{i}(y_i)-y_i} \leq \frac{1}{2(m+1)} \right),
	\end{equation*}
	and the result follows by Lemma~\ref{lemmaresolvineq}.
\end{proof}

\begin{lemma}\label{lemma_g}
	Let $D \in \N \setminus \{ 0 \}$ be such that $D\geq \max\{2 \norm{u-p},\norm{z_0-p} \}$, for some $p \in S$, and $c\in \N \setminus \{0\} $ satisfying $(Q_4)$. For $m,n \in \N$, $f:\N \to \N$ monotone and $z \in B_{D}$, if  $n \geq \ell ((r_1+1)8D^2-1)$ and
	\begin{equation*}
	\forall i \in [n,\phi_1(r_1,n,f+1)] \left( s_{i+1}^{z }\leq s_i^{z} \wedge P_i^{z} \leq \frac{1}{r_1+1}\right),
	\end{equation*}
	then
	\begin{equation*}
	\exists n' \leq \phi_2 (r_1,n,f+1)\, \forall i \in [n',f(n')] \left(\norm{J(y_i)-y_i}\leq \frac{1}{m+1}\right),
	\end{equation*}
	where $r_1:=r_1(m)=12c(m+1)^2-1$, as in Definition~\ref{definitionfunctions1}.
\end{lemma}

\begin{proof}
	We may assume that $f(n) \geq n$, because otherwise the result is trivial. By Lemma~\ref{lemmametadifference}, there exists $n' \in [n,\phi_{2}(r_1,n,f+1)]$ such that
	\begin{equation*}
	\forall i \in [n', f(n')] \left(s_{i}^{z}-s_{i+1}^{z} \leq \frac{1}{12c(m+1)^2} \right).
	\end{equation*}
	Since $n'\geq n$ and $f(n')\leq \phi_1(r_1, n, f+1)$ we have that $n' \geq \ell((r_1+1)8D^2-1)$ and $[n',f(n')]\subseteq [n, \phi_1(r_1,n,f+1)],$ which implies that $n' \geq \ell (96cD^2(m+1)^2-1)$ and 
	\begin{equation*}
	\forall i \in [n',f(n')] \left(s_{i}^{z}-s_{i+1}^{z}\leq \frac{1}{r_1+1} \wedge P_{i}^{z}\leq \frac{1}{r_1+1} \right).
	\end{equation*}
	Hence, by Lemma~\ref{linnerproduct0},	
\begin{equation*}
	\forall i \in [n', f(n')] \left(\norm{J(y_i)-y_i}\leq \frac{1}{m+1} \right).\qedhere
	\end{equation*}
\end{proof}

The analysis of the first case is concluded with the following result. It gives a rate of metastability for the convergence of the sequence $(s_n^{z})$ provided that $z$ is a suficiently good approximation to the projection point and that the decreasing property of the sequence $(s_n^{z})$ holds long enough.

\begin{lemma}\label{maincase1}
	Let $D \in \N \setminus \{ 0 \}$ be such that $D\geq \max\{2 \norm{u-p},\norm{z_0-p} \}$, for some $p \in S$. For $k,m \in \N$, $f:\N \to \N$ monotone and $z \in B_{D}$, assume that
	\begin{enumerate}[$(i)$]
	\item\label{lc13} $\forall i\in [r_3, \phi_1(r_1, r_3, \widetilde{f}+2)]\, \left( s_{i+1}^z\leq s_i^z \right)$,
	\item\label{lc11} $\norm{J(z)-z}\leq \dfrac{1}{\Xi_1(m)+1}$,
	\item\label{lc12} $\forall y\in B_D\, \left( \norm{J(y)-y}\leq \frac{1}{m+1}\to \langle u-z, y-z\rangle \leq \frac{1}{6(k+1)} \right)$.
	\end{enumerate}
	Then
	\begin{equation*}
	\exists n\leq \overline{\sigma}(k,\phi_2(r_1,r_3,\widetilde{f}+2))\, \forall i\in[n, f(n)]\, \left(s_i^z\leq \frac{1}{k+1} \right),
	\end{equation*}
where  $\overline{\sigma}$, $\widetilde{f}$, $r_1:=r_1(m)$, $r_3:=r_3(m)$, $\phi_1$ and $\phi_2$ are as in Definition~\ref{definitionfunctions1} and $\Xi_1$ is as Definition~\ref{definitionfunctions2}.
\end{lemma}
\begin{proof}
Let $r_4:=r_4(m)$ be as in Definition~\ref{definitionfunctions1}. By \eqref{lc11}, using Lemma~\ref{samefix-pt-set} and the definition of $\Xi_1$, we have
	\begin{equation*}
	\forall i\leq \phi_1(r_1,r_3,\widetilde{f}+2)\, \left(\norm{J_i(z)-z}\leq \frac{1}{2(6D+1)\max \{r_1,r_4\}} \right).
	\end{equation*}
	Noticing that $\norm{J_i(z)-z}\leq 1$ and $\norm{y_i-z}\leq 2D$, for $i\leq \phi_1(r_1,r_3,\widetilde{f}+2)$, we have
	\begin{equation}\label{eqr1r2}
	P_{i}^{z}=2\norm{J_{i}(z)-z}\left(3\norm{y_i-z}+\norm{J_{i}(z)-z} \right)\leq  \frac{2(6D+1)}{2(6D+1)\max \{r_1,r_4\}}=\frac{1}{\max \{r_1,r_4\}}.
	\end{equation}	
By	\eqref{eqr1r2}, \eqref{lc13} and the fact that $r_3 \geq \ell(96cD^2(m+1)^2-1)$, it follows from Lemma~\ref{lemma_g} that 
\begin{equation}\label{eqJyim}
\exists n \leq \phi_2(r_1,r_3,\widetilde{f}+2) \, \forall i \in [n,\widetilde{f}(n)+1] \left(\norm{J(y_i)-y_i}\leq \frac{1}{m+1} \right).
\end{equation} 
For $i \in [n, \widetilde{f}(n)]$, we have that $i+1 \in [n,\widetilde{f}(n)+1]$. Hence $\norm{J(y_{i+1})-y_{i+1}}\leq \frac{1}{m+1}$. It follows from \eqref{lc12} and the fact that $y_{i+1} \in B_{D}$ that
\begin{equation}\label{eqinnerproduct}
\exists n \leq \phi_2(r_1,r_3,\widetilde{f}+2) \, \forall i \in [n, \widetilde{f}(n)] \left(\langle u-z, y_{i+1}-z \rangle \leq \frac{1}{6(k+1)}\right).
\end{equation}
Since $\widetilde{f}(n)\leq \widetilde{f}(\phi_2(r_1,r_3,\widetilde{f}+2))$, we have $r_4 \geq 3(k+1)(\widetilde{f}(n)+1)$. Since $\widetilde{f}(n)\leq  \phi_1(r_1,r_3,\widetilde{f}+2)$, by \eqref{eqr1r2} and \eqref{eqinnerproduct} it follows that 
\begin{equation}\label{eqinnerproductPsmall}
\exists n \leq \phi_2(r_1,r_3,\widetilde{f}+2) \, \forall i \in [n, \widetilde{f}(n)] \left(P_{i}^{z}\leq \frac{1}{3(k+1)(\widetilde{f}(n)+1)} \wedge \langle u-z, y_{i+1}-z \rangle \leq \frac{1}{6(k+1)}\right).
\end{equation}
Then from \eqref{eqinnerproductPsmall}, by applying Lemma~\ref{lemmaqtXu1} with $q=\widetilde{f}(n):=f(\overline{\sigma}(k,n))$ and using inequality \eqref{Desigmain} we conclude
\begin{equation*}
	\exists n\leq \phi_2(r_1,r_3,\widetilde{f}+2)\, \forall i\in[\overline{\sigma}(k,n), f(\overline{\sigma}(k,n))]\, \left(s_i^z\leq \frac{1}{k+1} \right),
	\end{equation*}
which entails the result.
\end{proof}

\subsection{Second case}
We are now going to consider the case where the sequence $(s^{z}_n)$ is not eventually decreasing. 

For $s:\N \to \N$ and $m \in \N$ we define a functional $\tau$ as follows.
\begin{equation*}
\tau^{s}_{m}(n):=
\begin{cases} n & n<m\\   
				\max\{ k \in [m,n]: s_{k}<s_{k+1}\} & n \geq m \wedge\exists k \in [m,n] \left(s_{k}<s_{k+1}\right)  \\
			    n   & n \geq m \wedge  \forall k \in [m,n] \left(s_{k+1}\leq s_{k}\right) . 
\end{cases}
\end{equation*}

\begin{remark}
Given $s:\N \to \N$ and $m \in \N$, the definition of $\tau^{s}_{m}$ implies immediately that for all $i \in \N$ it holds that $\tau^{s}_{m}(i) \leq i$. Moreover, $\tau^{s}_{m}(m)=m$ and for $n \in \N$, if $s_{j+1} \leq s_{j}$, for all $j \in [m,n]$, then $\tau^{s}_{m}(i) =i$, for all $i\leq n$.
\end{remark}
The following proposition shows that the functional $\tau^{s}_{m}$ is monotone in  $m$ and $n$, respectively.

\begin{proposition}\label{proptau}
Let $s:\N \to \N$ and $m \in \N$. The functional $\tau^{s}_{m}$ enjoys the following properties
\begin{enumerate}[$(i)$]
\item\label{tau3}  $\forall i \in \N \left(\tau^{s}_{m}(i) \leq \tau^{s}_{m}(i+1) \right)$.
\item\label{tau4} $\forall i \in \N \left(\tau^{s}_{m}(i) \leq \tau^{s}_{m+1}(i) \right)$.
\end{enumerate}
\end{proposition}

\begin{proof}
\eqref{tau3}. If $i+1\leq m$, then $\tau^{s}_{m}(i) \leq i+1 =\tau^{s}_{m}(i+1)$. Assume that $i+1>m$. Then $i \geq m$. If $\exists j \in [m,i](s_j <s_{j+1})$, then clearly $\exists j \in [m,i+1](s_j <s_{j+1})$ and 
\begin{equation*}
\tau^{s}_{m}(i)= \max \{ j \in [m,i]: s_j <s_{j+1}\} \leq  \max \{ j \in [m,i+1]: s_j <s_{j+1}\}= \tau^{s}_{m}(i+1).
\end{equation*}
On the other hand, if $\forall j \in [m,i](s_{j+1}\leq s_j)$ then $\tau^{s}_{m}(i)=i \leq i+1 =\tau^{s}_{m}(i+1)$.

\eqref{tau4}. If $i \leq m$, then $\tau^{s}_{m}(i)=i=\tau^{s}_{m+1}(i)$. If $i = m+1$, then $\tau^{s}_{m}(m+1) \leq m+1 = \tau^{s}_{m+1}(m+1)$. If $i> m+1$ and $\forall j \in [m,i](s_{j+1}\leq s_{j})$, then clearly also $\forall j \in [m+1,i](s_{j+1}\leq s_{j})$ and $\tau^{s}_{m}(i)=i=\tau^{s}_{m+1}(i)$. If $i> m+1$ and $\exists j \in [m,i](s_j< s_{j+1})$, we have that either 
\begin{equation}\label{condtodo}
\forall j \in [m+1,i](s_{j+1}\leq s_{j})
\end{equation}
or
\begin{equation}\label{condexiste}
\exists j \in [m+1,i](s_j< s_{j+1}).
\end{equation}
If \eqref{condtodo} holds, then we must have $s_m< s_{m+1}$ and hence  $\tau^{s}_{m}(i)=m < i = \tau^{s}_{m+1}(i)$. If \eqref{condexiste} holds, then 
\begin{equation*}
\tau^{s}_{m}(i)= \max \{ j \in [m,i]: s_j <s_{j+1}\} =  \max \{ j \in [m+1,i]: s_j <s_{j+1}\}= \tau^{s}_{m+1}(i). \qedhere
\end{equation*}
\end{proof}

We recall that the original proof relies on \cite[Lemma~3.1]{Mainge}. As it turns out, for our quantitative analysis we do not need a full quantitative version of that result as the following weakening is sufficient.

\begin{lemma}\label{qtlemmaMainge}
Let $s:\N \to \N$ and $m,r \in \N$ be arbitrary. If $m \geq r$ and $s_{m}<s_{m+1}$, then
\begin{equation*}
 \forall i \geq m\left(\tau^{s}_{m}(i) \geq r \wedge \max\{s_{\tau^{s}_{m}(i)},s_i\}\leq s_{\tau^{s}_{m}(i)+1} \right) .
\end{equation*}
\end{lemma}

\begin{proof}
We show first that
\begin{equation}\label{tau5}
s_{m}<s_{m+1} \to \forall i \geq m\left(\max\{s_{\tau^{s}_{m}(i)},s_i\}\leq s_{\tau^{s}_{m}(i)+1} \right)
\end{equation}
Assume that $s_{m}<s_{m+1}$ and take arbitrary $i \geq m$. Then $\tau^{s}_{m}(i)=\max \{j \in [m,i]:  s_{j}<s_{j+1}\}$ and, in particular, $s_{\tau^{s}_{m}(i)} \leq s_{\tau^{s}_{m}(i)+1}$. On the other hand, consider the following three cases: (i) $\tau^{s}_{m}(i)=i$, (ii) $\tau^{s}_{m}(i)=i-1$ and (iii) $\tau^{s}_{m}(i)< i-1$. (i) From $s_{\tau^{s}_{m}(i)}\leq s_{\tau^{s}_{m}(i)+1}$, we get $s_i \leq s_{\tau^{s}_{m}(i)+1}$. (ii) This case reduces to $s_i \leq s_i$ which is trivially true. (iii) Note that for $j \in [\tau^{s}_{m}(i)+1, i-1]$ we must have $s_{j+1} \leq s_{j}$. Hence  $s_i \leq s_{i-1}\leq \dots \leq s_{\tau^{z}_{m}(i)+1}$. This finishes the proof of \eqref{tau5}.

Assume that $m \geq r $. Then by Part~\eqref{tau3} of Proposition~\ref{proptau}, for $i \geq m$, it follows that $\tau^{s}_{m}(i) \geq \tau^{s}_{m}(m)=m \geq r$, which concludes the proof.
\end{proof}

In the following, since we are going to use sequences $(s^z_n)$ that involve a parameter $z \in H$, we simplify the notation writting $\tau^{z}_{m}$ instead of $\tau^{s^z}_{m}$.

\begin{lemma}\label{lemmamaincase2}
Let $D \in \N \setminus \{ 0 \}$ be such that $D\geq \max\{2 \norm{u-p},\norm{z_0-p} \}$, for some $p \in S$. For $k,m,n \in \N$, $f:\N \to \N$  monotone and $z \in B_{D}$, assume that 
\begin{enumerate} [$(i)$]
\item\label{quantcase2} $n \geq r_3(m) \wedge s_n^{z} <s_{n+1}^{z}$,
\item\label{IneqJz-z} $\norm{J(z)-z}\leq \frac{1}{\xi(n)+1}$,
\item\label{hypprodint}
$\forall y \in B_{D} \left(\norm{J(y)-y}\leq \frac{1}{m+1} \to \langle u-z,y-z \rangle \leq \frac{1}{32(k+1)^2} \right).$
\end{enumerate}
Then $$\forall i \in [n,f(n)]\left(s_i^z \leq \frac{1}{k+1} \right),$$
where $\zeta$, $r_2$ and $r_3$ as in Definition~\ref{definitionfunctions1} and $\xi$ is as in Definition~\ref{definitionfunctions2}. 
\end{lemma}

\begin{proof}
In this proof we omit the parameter $z$, whenever possible, and write $s_{(\cdot)}$, $\tau_n(\cdot)$ and $P_{(\cdot)}$ instead of $s_{(\cdot)}^{z}$, $\tau_n^{z}(\cdot)$ and $P_{(\cdot)}^{z}$, respectively. 

By Lemma~\ref{qtlemmaMainge} we have that 
\begin{equation*}
\forall i \geq n \left( \tau_{n}(i) \geq r_3(m) \wedge  \max\{s_{\tau_{n}(i)},s_i\}\leq s_{\tau_{n}(i)+1} \right).
\end{equation*}

Let $i \in [n,f(n)]$. Since $s_{\tau_{n}(i)} \leq s_{\tau_{n}(i)+1}$, by inequality \eqref{desigJnyn-yn} we have that
\begin{equation}\label{eq10caso2}
\norm{J_{\tau_{n}(i)}(y_{\tau_{n}(i)})-y_{\tau_{n}(i)}}^2 \leq 8cD^2\lambda_{\tau_{n}(i)}+cP_{\tau_{n}(i)}.
\end{equation} 
By the monotonicity of $\ell$ and the definition of $r_3(m)$ we have that $\tau_{n}(i) \geq \ell\left(16c(r_2(m))^2D^2\right)$. Hence, by $(Q_{\ref{Q1}})$ 
 \begin{equation}\label{Ineqlambda}
8cD^2\lambda_{\tau_{n}(i)}\leq \frac{8cD^2}{16c(r_2(m))^2D^2}=\frac{1}{2(r_2(m))^2}.
\end{equation} 
By \eqref{IneqJz-z}, the monotonicity of $\zeta$ and the definition of $\xi$ we have that $\norm{J(z)-z}\leq \frac{1}{\zeta\left(4c(r_2(m))^2(6D+1),f(n)\right)+1}$. Since $\tau_{n}(i)\leq i \leq f(n)$, by Lemma~\ref{samefix-pt-set}  we have
\begin{equation*}
\norm{J_{\tau_{n}(i)}(z)-z}\leq\frac{1}{4c(r_2(m))^2(6D+1)}\, (\leq 1).
\end{equation*}
Then, since $\norm{y_{\tau_{n}(i)}-z}\leq 2D$, 
\begin{equation}\label{IneqP}
P_{\tau_{n}(i)}=2\norm{J_{\tau_{n}(i)}(z)-z}\left(3\norm{y_{\tau_{n}(i)}-z}+\norm{J_{\tau_{n}(i)}(z)-z}\right)\leq \frac{2(6D+1)}{4c(r_2(m))^2(6D+1)}=\frac{1}{2c(r_2(m))^2}.
\end{equation}
Combining $\eqref{eq10caso2} - \eqref{IneqP}$ we conclude that
\begin{equation*}
\norm{J_{\tau_{n}(i)}(y_{\tau_{n}(i)})-y_{\tau_{n}(i)}}^2 \leq \frac{1}{2(r_2(m))^2}+c\frac{1}{2c(r_2(m))^2}=\frac{1}{(r_2(m))^2}.
\end{equation*}
By the definition of $r_2(m)$ we conclude that
\begin{equation}\label{ineqJtau}
\norm{J_{\tau_{n}(i)}(y_{\tau_{n}(i)})-y_{\tau_{n}(i)}} \leq \dfrac{1}{128D(k+1)^2}.
\end{equation}
 Moreover, the definition of $r_2(m)$ and Lemma~\ref{lemmaresolvineq} entail that $\norm{J(y_{\tau_{n}(i)})-y_{\tau_{n}(i)}}\leq \frac{1}{m+1}.$

We show that 
\begin{equation}\label{ineqQ}
P_{\tau_{n}(i)} \leq \frac{1}{8h(i)(k+1)^2}.
\end{equation}
Indeed, by the definition of $\xi$ we have $\norm{J(z)-z}\leq \frac{1}{\zeta(16h(f(n))(k+1)^2(6D+1),f(n))+1}$. Then,  by Lemma~\ref{samefix-pt-set}, for $n'\leq f(n) $ $$\norm{J_{n'}(z)-z}\leq \frac{1}{16h(f(n))(k+1)^2(6D+1)}.$$
Since $h$ is monotone and $i\leq f(n)$ we have $\tau_{n}(i)\leq f(n)$ and $h(i)\leq h(f(n))$. Then $$\norm{J_{\tau_{n}(i)}(z)-z}\leq \frac{1}{16h(i)(k+1)^2(6D+1)}(\leq 1).$$
Hence
\begin{equation*}
P_{\tau_{n}(i)} \leq \frac{2(6D+1)}{16h(i)(k+1)^2(6D+1)}=\frac{1}{8h(i)(k+1)^2}.
\end{equation*}

By $(Q_{\ref{Q1}})$ and the the fact that $\tau_{n}(i) \geq r_3(m)$ we have that $\lambda_{\tau_{n}(i)} \leq \frac{1}{256D^2(k+1)^2}.$ Then, the definition of $\eqref{exactPPA}$ and the fact that $\norm{y_{\tau_{n}(i)}-u}\leq 2D$ entail
\begin{equation}\label{designormytau}
\begin{split}
\norm{y_{\tau_{n}(i)}-y_{\tau_{n}(i)+1}} &=\norm{\lambda_{\tau_{n}(i)}(y_{\tau_{n}(i)}-u)+\delta_{\tau_{n}(i)}\left(y_{\tau_{n}(i)}-J_{\tau_{n}(i)}(y_{\tau_{n}(i)}\right)}\\
&\leq \lambda_{\tau_{n}(i)}\norm{y_{\tau_{n}(i)}-u}+ \delta_{\tau_{n}(i)} \norm{J_{\tau_{n}(i)}(y_{\tau_{n}(i)})-y_{\tau_{n}(i)}} \\
& \leq 2D\lambda_{\tau_{n}(i)} +\norm{J_{\tau_{n}(i)}(y_{\tau_{n}(i)})-y_{\tau_{n}(i)}}.
\end{split}
\end{equation}
Hence, using \eqref{ineqJtau}, we derive that $\norm{y_{\tau_{n}(i)}-y_{\tau_{n}(i)+1}} \leq \frac{2D}{256D^2(k+1)^2}+\frac{1}{128D(k+1)^2} =\frac{1}{64D(k+1)^2},$ which trivialy implies that
\begin{equation}\label{ineqprinc2}
\norm{y_{\tau_{n}(i)}-y_{\tau_{n}(i)+1}} \leq \frac{1}{2(k+1)}.
\end{equation}

Since $y_{\tau_{n}(i)} \in B_D$ and $\norm{J(y_{\tau_{n}(i)})-y_{\tau_{n}(i)}} \leq \frac{1}{m+1}$, from \eqref{hypprodint} we obtain
\begin{equation}\label{desigprodint}
\langle u-z,y_{\tau_{n}(i)}-z \rangle \leq \frac{1}{32(k+1)^2}.
\end{equation}

We have that $\norm{u-z}\leq 2D$. Then, using \eqref{desigprodint}
\begin{equation*}
\begin{split}
\langle u-z, y_{\tau_{n}(i)+1}-z\rangle &= \langle u-z, y_{\tau_{n}(i)+1}-y_{\tau_{n}(i)}\rangle+\langle u-z, y_{\tau_{n}(i)}-z\rangle \\
& \leq \norm{u-z} \norm{y_{\tau_{n}(i)+1}-y_{\tau_{n}(i)}}+\langle u-z, y_{\tau_{n}(i)}-z\rangle \\
& \leq \frac{1}{16(k+1)^2}.
\end{split}
\end{equation*}

By \eqref{desig10}, using \eqref{ineqQ}, condition $(Q_1)$ and the fact that $h\left(\tau_{n}(i)\right)\leq h(i)$, we derive 
\begin{equation}\label{desigstn}
\begin{split}
s_{\tau_{n}(i)} &\leq 2 \langle u-z,y_{\tau_{n}(i)+1}-z \rangle +h\left(\tau_{n}(i)\right)\left(s_{\tau_{n}(i)}-s_{\tau_{n}(i)+1}+P_{\tau_{n}(i)}\right)\\
&\leq 2 \langle u-z,y_{\tau_{n}(i)+1}-z \rangle +h\left(\tau_{n}(i)\right)P_{\tau_{n}(i)}\\
&\leq \frac{1}{4(k+1)^2}.
\end{split}
\end{equation}

Observe that
\begin{equation}\label{desigsqrt}
\begin{split}
\sqrt{s_{\tau_{n}(i)+1}} &= \norm{y_{\tau_{n}(i)}-z-\left(y_{\tau_{n}(i)}-y_{\tau_{n}(i)+1}\right)}\\
&\leq \sqrt{s_{\tau_{n}(i)}}+ \norm{y_{\tau_{n}(i)}-y_{\tau_{n}(i)+1}}.
\end{split}
\end{equation}

Then, by \eqref{ineqprinc2}, \eqref{desigstn}  and \eqref{desigsqrt} we have

\begin{equation*}
\sqrt{s_{\tau_{n}(i)+1}} \leq \sqrt{\frac{1}{4(k+1)^2}}+\frac{1}{2(k+1)}=\frac{1}{k+1}.
\end{equation*}
Hence $s_{\tau_{n}(i)+1} \leq \frac{1}{(k+1)^2} \leq \frac{1}{k+1}$, which entails the result.
\end{proof}

\subsection{Putting it together}
We are now able to prove our main result.

\begin{proof}[Proof of Theorem~\ref{theoremwangcui}]
By Proposition~\ref{lemmaprojectarg} there exist $m_0 \leq \beta(\overline{k},\Xi)$ and $z \in B_{D}$ such that 
\begin{equation*}
\norm{J(z)-z}\leq \frac{1}{\Xi(m_0)+1}
\end{equation*}
and 
\begin{equation*}
\forall y \in B_{D} \left(\norm{J(y)-y}\leq \frac{1}{m_0+1} \to \langle u-z,y-z\rangle \leq \frac{1}{32(k+1)^2} \right). 
\end{equation*}

Consider $r_1$ and  $r_3$ to be, respectively, the natural numbers $r_1(m_0)$ and $r_3(m_0)$. 
Observe that  by monotonicity (cf. Remark~\ref{remrakmonotone}),  $r_1\leq \bar{r}$ and $r_3\leq \bar{n}$.

We may assume that $\widetilde{f}(r_3)\geq r_3$. Indeed, if $f(\overline{\sigma}(k,r_3))<r_3$, then $f(r_3)<r_3$ by monotonicity and the fact that $\overline{\sigma}(k,r_3)\geq r_3$ (cf. Remark~\ref{remarksigma}). Notice that by the definition of $\phi_2$ we obtain $r_3 \leq \overline{n}\leq \phi_2(\overline{r}, \overline{n}, \widetilde{f}+2)$. Again by monotonicity and the fact that $\overline{\sigma}(k,r_3)\geq r_3$ we would then have that $r_3 \leq \mu(k,f)$ and the result would be trivially true. 
The condition $\widetilde{f}(r_3)\geq r_3$ implies that $\phi_2(r_1,r_3,\widetilde{f}+2) \leq \phi_2(\overline{r},\overline{n},\widetilde{f}+2)$ and consequently $\phi_1(r_1,r_3,\widetilde{f}+2) \leq \phi_1(\overline{r},\overline{n},\widetilde{f}+2)$. 

 If
\begin{equation*}
\forall i \in [r_3, \phi_1(r_1,r_3,\widetilde{f}+2)] \left(s_{i+1}^{z}<s_{i}^{z}\right),
\end{equation*}
then by Lemma~\ref{maincase1}, there is $n\leq \overline{\sigma}(k, \phi_2(r_1,r_3,\widetilde{f}+2)) \leq \mu(k,f)$ such that
\begin{equation*}
\forall i \in [n,f(n)] \left(s_{i}^{z}\leq \frac{1}{k+1} \right).
\end{equation*}
On the other hand, if $s_{n}^{z}\leq s_{n+1}^{z}$ for some $n \in [r_3, \phi_1(r_1,r_3,\widetilde{f}+2)]=[r_3, \Phi(m_0)]$, we have 
\begin{equation*}
\norm{J(z)-z}\leq \frac{1}{\Xi(m_0)+1}\leq \frac{1}{\xi(\Phi (m_0))+1}\leq \frac{1}{\xi(n)+1}.
\end{equation*}
By Lemma~\ref{lemmamaincase2}, we conclude that there is $n\leq \Phi(m_0)\leq \mu(k,f)$ such that
\begin{equation*}
\forall i \in [n,f(n)] \left(s_{i}^{z}\leq \frac{1}{k+1} \right). \qedhere
\end{equation*}
\end{proof}

\section{Final remarks}\label{sectionremarks}

We observe that conditions $(Q_1) - (Q_{4})$ together with either condition $(Q_{5a})$ or $(Q_{5b})$ allow the sequence $(\gamma_n)$ to be identically equal to zero and so, by taking that choice for $(\gamma_n)$, the iteration \eqref{PPA} reduces to \eqref{HPPA}. In that case, condition $(Q_4)$ can be written as $\forall n \in \N \left(\min\{c_n, (1-\lambda_n)^2\} \geq \frac{1}{c}\right)$ and a quantitative version holds with the same bounds. In fact, that quantitative version is a generalization of previous analyses \cite{LLPP(ta),PP(ta)}, as Theorem~\ref{ThmWangCui} has weaker conditions than those of \cite[Theorem~5.1]{X(02)} and \cite[Theorem~2]{BM(11)}. However,  the analyses in \cite{LLPP(ta),PP(ta)} are still of interest as the bounds obtained there are much simpler than the ones obtained in this paper.  

Under the quantitative conditions  $(Q_1) - (Q_{4})$ together with either $(Q_{5a})$ or $(Q_{5b})$, in corollaries \ref{cor_metayn} and \ref{cor_metazn} we gave explicit bounds on the metastability of the iterations \eqref{PPA} and \eqref{exactPPA}, and in corollaries \ref{cor_metaJi} and \ref{cor_metaJiz} we computed a bound on (the metastable version of) the asymptotic regularity of these iterations. Let us argue that these results provide a quantitative version of Theorem~\ref{ThmWangCui}. By Corollary~\ref{cor_metayn} it follows (ineffectively) that \eqref{exactPPA} is a Cauchy sequence. Hence it converges strongly to a point $\widetilde{y} \in H$. By Corollary~\ref{cor_metaJi} and the continuity of the resolvent functions it follows that $\widetilde{y}$ must be a fixed point, and therefore a zero of the operator $\mathsf{A}$. Furthermore, one can argue that $\widetilde{y}$ must be the projection point of $u$ onto $S$. Indeed, consider the sequence $\left(s_n^{\widetilde{z}}\right)$ with $\widetilde{z}$ a projection point. One can argue, as in Lemmas~\ref{maincase1} and \ref{lemmamaincase2}, to conclude that Theorem~\ref{theoremwangcui} holds with $z=\widetilde{z}$ for every $k$ and $f$. Notice that one cannot guarantee the third assumption in neither of those lemmas. However, those conditions are only required to show equations \eqref{eqinnerproduct} and \eqref{desigprodint}, respectively, which follow from the fact that $\langle u- \widetilde{z}, \widetilde{y}-\widetilde{z} \,\rangle \leq 0$ and the fact that $\widetilde{y}= \lim y_n$. Since Theorem~\ref{theoremwangcui} is always true with $z=\widetilde{z}$, we conclude that $\widetilde{y}$ must be the projection point. By Lemma~\ref{lemma_zn-yn_small} one concludes that the iteration $\eqref{PPA}$ must also converge strongly to a zero of the operator, namely the projection point.

\section*{Acknowledgements}
Both authors acknowledge the support of FCT - Funda\c{c}\~ao para a Ci\^{e}ncia e Tecnologia under the project: UID/MAT/04561/2019 and the research center Centro de Matem\'{a}tica, Aplica\c{c}\~{o}es Fundamentais e Investiga\c{c}\~{a}o Operacional, Universidade de Lisboa. 
The second author also acknowledges the support of the `Future Talents' short-term scholarship at Technische Universit{\"a}t Darmstadt.

The paper also benefited from discussions with Fernando Ferreira and Ulrich Kohlenbach.

\bibliography{References}{}
\bibliographystyle{plain}

\end{document}